\documentclass[12pt]{amsart}
\usepackage{mathrsfs}
\usepackage{amssymb}
\usepackage{verbatim}
\usepackage{hyperref}
\usepackage{color}
\usepackage{amsfonts}
\usepackage{mathrsfs}
\usepackage{amsmath}
\usepackage{amssymb}

\begin{document}
\newtheorem{theorem}{Theorem}
\newtheorem{proposition}[theorem]{Proposition}
\newtheorem{conjecture}[theorem]{Conjecture}
\def\theconjecture{\unskip}
\newtheorem{corollary}[theorem]{Corollary}
\newtheorem{lemma}[theorem]{Lemma}
\newtheorem{sublemma}[theorem]{Sublemma}
\newtheorem{observation}[theorem]{Observation}
\theoremstyle{definition}
\newtheorem{definition}{Definition}
\newtheorem{notation}[definition]{Notation}
\newtheorem{remark}[definition]{Remark}
\newtheorem{question}[definition]{Question}
\newtheorem{questions}[definition]{Questions}
\newtheorem{example}[definition]{Example}
\newtheorem{problem}[definition]{Problem}
\newtheorem{exercise}[definition]{Exercise}

\numberwithin{theorem}{section} \numberwithin{definition}{section}
\numberwithin{equation}{section}

\def\earrow{{\mathbf e}}
\def\rarrow{{\mathbf r}}
\def\uarrow{{\mathbf u}}
\def\varrow{{\mathbf V}}
\def\tpar{T_{\rm par}}
\def\apar{A_{\rm par}}

\def\reals{{\mathbb R}}
\def\torus{{\mathbb T}}
\def\heis{{\mathbb H}}
\def\integers{{\mathbb Z}}
\def\naturals{{\mathbb N}}
\def\complex{{\mathbb C}\/}
\def\distance{\operatorname{distance}\,}
\def\support{\operatorname{support}\,}
\def\dist{\operatorname{dist}\,}
\def\Span{\operatorname{span}\,}
\def\degree{\operatorname{degree}\,}
\def\kernel{\operatorname{kernel}\,}
\def\dim{\operatorname{dim}\,}
\def\codim{\operatorname{codim}}
\def\trace{\operatorname{trace\,}}
\def\Span{\operatorname{span}\,}
\def\dimension{\operatorname{dimension}\,}
\def\codimension{\operatorname{codimension}\,}
\def\nullspace{\scriptk}
\def\kernel{\operatorname{Ker}}
\def\ZZ{ {\mathbb Z} }
\def\p{\partial}
\def\rp{{ ^{-1} }}
\def\Re{\operatorname{Re\,} }
\def\Im{\operatorname{Im\,} }
\def\ov{\overline}
\def\eps{\varepsilon}
\def\lt{L^2}
\def\diver{\operatorname{div}}
\def\curl{\operatorname{curl}}
\def\etta{\eta}
\newcommand{\norm}[1]{ \|  #1 \|}
\def\expect{\mathbb E}
\def\bull{$\bullet$\ }

\def\xone{x_1}
\def\xtwo{x_2}
\def\xq{x_2+x_1^2}
\newcommand{\abr}[1]{ \langle  #1 \rangle}

\def\blue{\color{blue}}
\def\red{\color{red}}

\newcommand{\Norm}[1]{ \left\|  #1 \right\| }
\newcommand{\set}[1]{ \left\{ #1 \right\} }
\def\one{\mathbf 1}
\def\whole{\mathbf V}
\newcommand{\modulo}[2]{[#1]_{#2}}
\def \essinf{\mathop{\rm essinf}}
\def\scriptf{{\mathcal F}}
\def\scriptg{{\mathcal G}}
\def\scriptm{{\mathcal M}}
\def\scriptb{{\mathcal B}}
\def\scriptc{{\mathcal C}}
\def\scriptt{{\mathcal T}}
\def\scripti{{\mathcal I}}
\def\scripte{{\mathcal E}}
\def\scriptv{{\mathcal V}}
\def\scriptw{{\mathcal W}}
\def\scriptu{{\mathcal U}}
\def\scriptS{{\mathcal S}}
\def\scripta{{\mathcal A}}
\def\scriptr{{\mathcal R}}
\def\scripto{{\mathcal O}}
\def\scripth{{\mathcal H}}
\def\scriptd{{\mathcal D}}
\def\scriptl{{\mathcal L}}
\def\scriptn{{\mathcal N}}
\def\scriptp{{\mathcal P}}
\def\scriptk{{\mathcal K}}
\def\frakv{{\mathfrak V}}
\def\C{\mathbb{C}}
\def\R{\mathbb{R}}
\def\Rn{{\mathbb{R}^n}}
\def\Sn{{{S}^{n-1}}}
\def\M{\mathbb{M}}
\def\N{\mathbb{N}}
\def\Q{{\mathbb{Q}}}
\def\Z{\mathbb{Z}}
\def\F{\mathcal{F}}
\def\L{\mathcal{L}}
\def\S{\mathcal{S}}
\def\ch{\operatorname{ch}}
\def\supp{\operatorname{supp}}
\def\dist{\operatorname{dist}}
\def\essi{\operatornamewithlimits{ess\,inf}}
\def\esss{\operatornamewithlimits{ess\,sup}}
\author{Mingming Cao}
\address{Mingming Cao \\
         School of Mathematical Sciences \\
         Beijing Normal University \\
         Laboratory of Mathematics and Complex Systems \\
         Ministry of Education \\
         Beijing 100875 \\
         People's Republic of China}
\email{m.cao@mail.bnu.edu.cn}

\author{Kangwei Li}
\address{Kangwei Li\\
        BCAM, Basque Center for Applied Mathematics, Mazarredo, 14. 48009
Bilbao, Basque Country, Spain}
\email{kli@bcamath.org}

\author{Qingying Xue}
\address{Qingying Xue\\
        School of Mathematical Sciences\\
        Beijing Normal University \\
        Laboratory of Mathematics and Complex Systems\\
        Ministry of Education\\
        Beijing 100875\\
        People's Republic of China}
\email{qyxue@bnu.edu.cn}

\thanks{M. Cao and Q. Xue were supported partly by NSFC
(No. 11471041 and 11671039), the Fundamental Research Funds for the Central Universities (NO. 2014KJJCA10) and NCET-13-0065. K. Li is
supported by the Basque Government through the BERC 2014-2017 program and by Spanish Ministry
of Economy and Competitiveness MINECO: BCAM Severo Ochoa excellence accreditation SEV-2013-
0323.\\ \indent Corresponding
author: Qingying Xue\indent Email: qyxue@bnu.edu.cn}


\date{March  1, 2017.}
\title[Littlewood-Paley $g_{\lambda}^{*}$-function]{\textbf{A Characterization of Two-weight Norm Inequality for Littlewood-Paley} $g_{\lambda}^{*}$-Function }
\maketitle

\begin{abstract}
Let $n\ge 2$ and $g_{\lambda}^{*}$ be the well-known high dimensional Littlewood-Paley function which was defined and studied by E. M. Stein,
\[
g_{\lambda}^{*}(f)(x)
=\bigg(\iint_{\R^{n+1}_{+}} \Big(\frac{t}{t+|x-y|}\Big)^{n\lambda}
|\nabla P_tf(y,t)|^2 \frac{dy dt}{t^{n-1}}\bigg)^{1/2}, \ \quad \lambda > 1,
\]
where $P_tf(y,t)=p_t*f(y)$, $p_t(y)=t^{-n}p(y/t)$ and $p(x) = (1+|x|^2)^{-(n+1)/2}$,
$\nabla =(\frac{\partial}{\partial y_1},\ldots,\frac{\partial}{\partial y_n},\frac{\partial}{\partial t})$.
In this paper, we give a characterization of two-weight norm inequality for $g_{\lambda}^{*}$-function.
We show that, $\big\| g_{\lambda}^{*}(f \sigma) \big\|_{L^2(w)} \lesssim \big\| f \big\|_{L^2(\sigma)}$ if and only if the two-weight Muchenhoupt $A_2$ condition holds, and a testing condition holds :
\[
\sup_{Q : cubes \ in \ \Rn} \frac{1}{\sigma(Q)} \int_{\Rn} \iint_{\widehat{Q}}
\Big(\frac{t}{t+|x-y|}\Big)^{n\lambda}|\nabla P_t(\mathbf{1}_Q \sigma)(y,t)|^2  \frac{w dx dt}{t^{n-1}} dy < \infty,
\]
where $\widehat{Q}$ is the Carleson box over $Q$ and $(w, \sigma)$ is a pair of weights. We actually prove this characterization for $g_{\lambda}^{*}$-function associated with more general fractional Poisson kernel $p^\alpha(x) = (1+|x|^2)^{-{(n+\alpha)}/{2}}$. Moreover, the corresponding results for intrinsic $g_{\lambda}^*$-function are also presented.
\end{abstract}


\section{Introduction}
The $g_{\lambda}^*$-function originated in the work of Littlewood and Paley \cite{LP} in the 1930's. It is a basic tool in the analysis of $L^p$ bounds for various linear operators. Later, the classical $g_{\lambda}^*$-function of higher dimension was first introduced and studied by Stein \cite{Stein1961} in 1961, a certain sublinear operator arises in Littlewood-Paley theory \cite{F}, \cite{Stein1970}. It plays important roles in Harmonic analysis and other fields. Let $n \geq 2$, we recall its definition as follows:
$$ g_{\lambda}^*(f)(x)=\bigg(\iint_{\R^{n+1}_{+}} \Big(\frac{t}{t+|x-y|}\Big)^{n\lambda}
|\nabla P_t f(y,t)|^2 \frac{dy dt}{t^{n-1}}\bigg)^{1/2},\quad\quad \lambda > 1$$
where $P_t f(y,t)=p_t*f(y)$, $p_t(y)=t^{-n}p(y/t)$, $p(y) = (1+|y|^2)^{-(n+1)/2}$ denotes the Poisson kernel and
$\nabla =(\frac{\partial}{\partial y_1},\ldots,\frac{\partial}{\partial y_n},\frac{\partial}{\partial t})$.
It is easy to show that $g_{\lambda}^*$ is an isometry on $L^2(\Rn)$. With much greater difficulty, it can be proved that for any $1< p < \infty$, $\big\| g_{\lambda}^*(f) \big\|_{L^p(\Rn)}$ and $\big\| f \big\|_{L^p(\Rn)}$ are equivalent norms \cite{Stein1961}. Moreover, in \cite{Stein1961}, Stein also proved that if $\lambda > 2$, then $g_{\lambda}^*$ is of weak type $(1,1)$, and is of strong type $(p,p)$ for $1 < p < \infty$.  In 1970, as a replacement of weak $(1,1)$ bounds for $1<\lambda<2$, Fefferman \cite{F} considered the end-point weak $(p,p)$ estimates of $g_{\lambda}^*$-function when $p>1$ and $\lambda=2/p$.

Recently, Lacey and the second named author \cite{LL} gave a characterization of two weight norm inequalities for the classical $g$-function and the corresponding intrinsic square function. Recall that the classical $g$-function is defined by
$$ g(f)(x)=\bigg(\int_{0}^{\infty} |\nabla P_t f(x,t)|^2 t dt \bigg)^{1/2}.$$
It was shown that the following two weight norm inequality for the classical Littlewood-Paley $g$-function for a pair of weights $(w, \sigma)$ on $\Rn$:
\begin{equation}\label{two weight g-function}
\big\| g(f \sigma) \big\|_{L^2(w)} \lesssim \big\| f \big\|_{L^2(\sigma)}
\end{equation}
holds if and only if $(w,\sigma)$ satisfies
\begin{equation}\label{A2 condition}
\mathscr{A}^2 := \sup_{Q} \frac{\sigma(Q)}{|Q|} \frac{w(Q)}{|Q|} < \infty;
\end{equation}
and the testing condition holds, uniformly over all cubes $Q \subset \Rn$,
\begin{equation}\label{testing condition-g}
 \iint_{\widehat{Q}} |\nabla P_t(\mathbf{1}_Q \sigma)(x,t)|^2 w dx \ t dt \lesssim \sigma(Q), \ \ \ \widehat{Q}=Q \times (0,\ell(Q)].
\end{equation}
The condition $(\ref{testing condition-g})$ is called the Sawyer testing condition, which can be traced back to \cite{Saywer}. It is known that Littlewood-Paley $g$-function is pointwise controlled by $g_\lambda^*$-function. Thus it is quite natural to ask if one can establish a characterization for the Littlewood-Paley $g_\lambda^*$-function. But the $g_\lambda^*$-function also involves additional difficulties since more integrals appear in the definition. One also needs to find the new suitable testing condition to replace condition $(\ref{testing condition-g})$.

In order to state our results, we first introduce the definition of the Littlewood-Paley $g_\lambda^*$-function with fractional Poisson kernels.

\begin{definition}\label{definition}
Let $\lambda > 1$, for any $x \in \Rn$, the Littlewood-Paley $g_\lambda^*$-function with fractional Poisson kernels is defined by
$$ g_{\lambda}^{*,\alpha}(f)(x)= \bigg(\iint_{\R^{n+1}_{+}} \Big(\frac{t}{t+|x-y|}\Big)^{n\lambda}
|\nabla P^\alpha_t f(y,t)|^2 \frac{dy dt}{t^{n-1}}\bigg)^{1/2},\ \ 0 < \alpha \leq 1,$$
where $P^\alpha_t f(y,t)=p^\alpha_t*f(y)$, $p^\alpha_t(y)=t^{-n}p^\alpha(y/t)$ and
$ p^\alpha(x) = (1+|x|^2)^{-(n+\alpha)/2}$.
\end{definition}
\begin{remark} If $\alpha=1,$ then $g_{\lambda}^{*,1}$ coincides with the classical Littlewood-Paley $g_{\lambda}^{*}$-function of higher dimension defined and studied by E. M. Stein
\cite{Stein1961} in 1961.
\end{remark}
Motivated by the above work, in this paper, we will focus on the characterization of the two weight inequality for the Littlewood-Paley $g_{\lambda}^{*}$-function.
\begin{equation}\label{t-w-i-g^star}
\big\| g_{\lambda}^{*,\alpha}(f \sigma) \big\|_{L^2(w)} \leq \mathscr{N} \big\| f \big\|_{L^2(\sigma)}.
\end{equation}
In addition, we introduce the corresponding testing condition:
\begin{equation}\label{testing condition-g^star}
\mathscr{B}^2 := \sup_{Q : \ cubes \ in \ \Rn} \frac{1}{\sigma(Q)} \iint_{\widehat{Q}} \int_{\Rn} \Big(\frac{t}{t+|x-y|}\Big)^{n\lambda}|\nabla P^\alpha_t(\mathbf{1}_Q \sigma)(y,t)|^2  dy
\frac{w dx dt}{t^{n-1}} < \infty .
\end{equation}

Here we formulate the main result of this paper as follows.
\begin{theorem}\label{Theorem g^star}
Let $\lambda > 2$, $ 0 < \alpha \leq \min\{1, n(\lambda - 2)/2 \}$ and $\sigma$, $w$ be two weights. Then the two weight inequality
$(\ref{t-w-i-g^star})$ holds if and only if the two weight condition $(\ref{A2 condition})$ and testing condition
$(\ref{testing condition-g^star})$ hold.
Moreover, $\mathscr{N} \simeq \mathscr{A} + \mathscr{B}$, where $\mathscr{N}$ is the best constant in the inequality $(\ref{t-w-i-g^star})$.
\end{theorem}

\begin{remark}
The characterization of the two weight inequality for the classical Littlewood-Paley $g_{\lambda}^*$-function is contained in Theorem $\ref{Theorem g^star}$ ($\alpha=1$, $\lambda \geq 2(1+{1}/{n})$). Actually, when $\lambda \geq 2(1+{1}/{n})$, we have $ 0 < \alpha \leq 1$. It not only includes  the classical case, but also extends to the case for $0 < \alpha < 1$. Another notable fact is that we are able to improve the result of \cite{LL} with the fractional Poisson kernel $p^\alpha$, $ 0 < \alpha \leq 1$.
\end{remark}

To state another main result, we begin with one more definition.
\begin{definition}\label{definition1}For $0 < \alpha \leq 1$, let $\mathcal{C}_{\alpha}$ be the family of functions $\varphi$ satisfying
$supp \ \varphi \subset \{ x \in \Rn; |x| \leq 1\}$, $\int_{\Rn} \varphi(x) dx = 0$, and such that
$|\varphi(x)-\varphi(x')| \leq |x-x'|^{\alpha}$, for all $x$, $x' \in \Rn$. If $f \in L^1_{loc}(\Rn)$ and $(y,t) \in \R_+^{n+1}$, we define
$$
A_{\alpha}f(y,t)=\sup_{\varphi \in \mathcal{C}_{\alpha}} |f*\varphi_t(y)|, \ \text{where}\ \varphi_t(x)=t^{-n} \varphi(x/t).
$$
Then the $intrinsic$ $g_{\lambda}^*$-function is defined by setting, for all $x \in \Rn$,
$$
g_{\lambda,\alpha}^*(f)(x)=\bigg(\iint_{\R^{n+1}_{+}} \Big(\frac{t}{t+|x-y|}\Big)^{n\lambda}
[A_{\alpha}f(y,t)]^2 \frac{dy dt}{t^{n+1}}\bigg)^{1/2}.
$$
\end{definition}
For the intrinsic $g_{\lambda,\alpha}^*$ function, we have the following result.

\begin{theorem}\label{Theorem g^star-alpha}
Let $\lambda > 2$, $ 0 < \alpha \leq \min\{1, n(\lambda - 2)/2 \}$ and $\sigma$, $w$ be two weights.
Then the two weight inequality
$$\big\| g_{\lambda,\alpha}^*(f \sigma) \big\|_{L^2(w)} \leq \mathscr{N}_\alpha \big\| f \big\|_{L^2(\sigma)}$$
holds if and only if
\begin{enumerate}
\item [(i)] $(w,\sigma)$ satisfies the $A_2$ condition $(\ref{A2 condition})$;
\item [(ii)] the testing condition holds :
$$
\mathscr{B}_{\alpha}^2
:= \sup_{Q:\ cubes \ in \ \Rn} \frac{1}{\sigma(Q)} \iint_{\widehat{Q}} \int_{\Rn} \Big(\frac{t}{t+|x-y|}\Big)^{n\lambda}[A_{\alpha}(\mathbf{1}_Q \sigma)(y,t)]^2  dy
\frac{w dx dt}{t^{n+1}} < \infty.
$$
\end{enumerate}
Moreover, the best constants satisfy $\mathscr{N}_{\alpha} \simeq \mathscr{A} + \mathscr{B}_{\alpha}$.
\end{theorem}

Note that $g_{\lambda}^{*,\alpha}f(x) \leq g_{\lambda,\alpha}^*f(x)$, for all $x \in \Rn$. Since the main steps in the proof of Theorem $\ref{Theorem g^star-alpha}$ are the same as the Theorem $\ref{Theorem g^star}$, we omit the proof of Theorem \ref{Theorem g^star-alpha}.

The rest of this article is organized as follows. The necessary condition is shown in the Section $\ref{Sec-Necessity}$. In Section $\ref{Sec-Reduction}$, applying the random dyadic grids and martingale difference decomposition, we give the final reduction of the main theorem. In order to prove the sufficiency, some lemmas and elementary estimates are established in Section $\ref{Sec-Lemmas}$. Finally, in Section $\ref{Sec-Sufficiency}$, by splitting into four parts, we prove the sufficiency in Theorem $\ref{Theorem g^star}$.

\vspace{0.2cm}
\textbf{Notation.} We write $A \lesssim B$, if there is a constant $C > 0$ so that $A \leq C B$. We may also write $A \simeq B$ if
$B \lesssim A \lesssim B$.

We then set some dyadic notation. For cubes $Q$ and $R$ we denote
\begin{enumerate}
\item [$\bullet$] $\ell(Q)$ is the side-length of $Q$;
\item [$\bullet$] $d(Q,R)$ denotes the distance between the cubes $Q$ and $R$;
\item [$\bullet$] $D(Q,R) := \ell(Q) + \ell(R) + d(Q,R)$ is the long distance;
\item [$\bullet$] $\widehat{Q} := Q \times (0,\ell(Q)]$ is the Carleson box over $Q$;
\item [$\bullet$] $W_Q := Q \times (\ell(Q)/2, \ell(Q)]$ is the Whitney region associated with $Q$;
\item [$\bullet$] $Q^{(k)}$ denotes the unique dyadic cube for which $\ell(Q^{(k)}) = 2^k \ell(Q)$ and $Q \subset Q^{(k)}$;
\item [$\bullet$] $\ch(Q)$ denotes the dyadic children of Q. More precisely, if the cube $Q=x+[0,\ell)^n$, then
$\ch(Q):=\big\{x+\eta \ell/2 +[0,\ell/2)^n; \eta \in \{0,1\}^n \big\}$.
\end{enumerate}

\section{The Necessity and constant estimates}\label{Sec-Necessity}
\subsection{Proposition} The inequality $(\ref{t-w-i-g^star})$ implies the inequality $(\ref{A2 condition})$.
\begin{proof}
For some fixed cube $Q$, we have
\begin{align*}
|\nabla P_t^\alpha (\mathbf{1}_Q \sigma)(y,t)|
\geq |\partial_t P_t^\alpha * (\mathbf{1}_Q \sigma)(y)| = \bigg| \int_Q \frac{n t^{\alpha + 1} - \alpha t^{\alpha - 1}|y-z|^2}
{(t^2 + |y-z|^2)^{\frac{n+\alpha}{2}+1}} \sigma dz \bigg|.
\end{align*}
If $x,y,z \in Q$ and $2 \ell(Q) \leq t \leq 3 \ell(Q)$, then
$$ \frac{n t^{\alpha + 1} - \alpha t^{\alpha - 1}|y-z|^2}{(t^2 + |y-z|^2)^{\frac{n+\alpha}{2}+1}} \gtrsim \frac{1}{t^{n+1}}.$$
Thus,
$$ |\nabla P_t^\alpha (\mathbf{1}_Q \sigma)(y,t)| \gtrsim \frac{\sigma(Q)}{t^{n+1}} .$$
Furthermore, for $x \in Q$,
\begin{align*}
g_\lambda^{*,\alpha}(\mathbf{1}_Q \sigma)(x)^2
&\geq \int_Q \int_{2 \ell(Q)}^{3 \ell(Q)}\Big(\frac{t}{t+|x-y|}\Big)^{n \lambda} \bigg| \int_Q \frac{n t^{\alpha + 1} -
\alpha t^{\alpha - 1}|y-z|^2}{(t^2 + |y-z|^2)^{\frac{n+\alpha}{2}+1}} \sigma dz \bigg|^2 \frac{dt}{t^{n-1}} dy \\
&\gtrsim \int_Q \int_{2 \ell(Q)}^{3 \ell(Q)} \frac{\sigma(Q)^2}{t^{3n+1}} dt dy
\gtrsim \frac{\sigma(Q)^2}{|Q|^2}.
\end{align*}
Therefore, the boundedness of $g_\lambda^{*,\alpha}$ gives that
$$\frac{\sigma(Q)w(Q)}{|Q|^2} \lesssim \frac{1}{\sigma(Q)} \big\| g_\lambda^{*,\alpha}(\mathbf{1}_Q \sigma) \big\|_{L^2(w)}^2 \leq \mathscr{N}^2.$$
That is, $\mathscr{A} \lesssim \mathscr{N}$.
\end{proof}
Moreover, it is trivial that $(\ref{t-w-i-g^star})$ implies $(\ref{testing condition-g^star})$. Thus, we have proved the necessity of Theorem $\ref{Theorem g^star}$.

\subsection{Random Dyadic Grids}
Let us recall random dyadic grids defined in \cite{H}. Denote by $\mathcal{D}=\mathcal{D}(\beta)$ the random dyadic grid, where $\beta=\{\beta_j\}_{j=-\infty}^{\infty} \in (\{0,1\}^n)^{\Z}$. That is,
$$ \mathcal{D}=\Big\{Q + \beta; Q \in \mathcal{D}_0 \Big\} := \Big\{Q + \sum_{j:2^{-j}<\ell(Q)} 2^{-j} \beta_j; Q \in \mathcal{D}_0 \Big\} ,$$
where $\mathcal{D}_0$ is the standard dyadic grid of $\Rn$.

\vspace{0.3cm}
\textbf{Good and Bad Cubes.}
A cube $I \in \mathcal{D}$ is said to be $bad$ if there exists a $J \in \mathcal{D}$ with $\ell(J) \geq 2^r\ell(I)$ such that
$\dist(I,\partial J) \leq \ell(I)^{\gamma} \ell(J)^{1-\gamma}$, where $r \in \Z_{+}$ and $\gamma \in (0,\frac12)$ are given parameters. Otherwise, $I$ is called $good$.

Throughout this article, we take $\gamma = \frac{\alpha}{2(n+\alpha)}$ and $r$ will be determined in the following. Moreover, roughly speaking, a dyadic cube $I$ will be bad if it is relatively close to the boundary of a much bigger dyadic cube. Denote
$\pi_{good} = \mathbb{P}_{\beta}(Q + \beta \ \text{is \ good})=\mathbb{E}_{\beta}(\mathbf{1}_{good}(Q + \beta))$.
Then $\pi_{good}$ is independent of $Q \in \mathcal{D}_0$. And we can choose $r$ large enough so that $\pi_{good}>0$.

\subsection{Definition}
Given a dyadic cube $I$, we set $\mathcal{W}_I$ to be the maximal dyadic cubes $K \subset I$ such that $2^r \ell(K) \leq \ell(I)$ and $\dist(K,\partial I) \geq \ell(K)^{\gamma} \ell(I)^{1-\gamma}$.

In order to meet the demands below, we present the following proposition, which was proved in \cite{LL}.
\subsection{Proposition}\label{overlap}
The following statements hold.
\begin{enumerate}
\item [(1)] For any good $J \Subset I$, there is a cube $K \in \mathcal{W}_I$ which contains $J$;
\item [(2)] For any $C > 0$, provided $r$ is sufficiently large, depending upon $\gamma$, there holds
$$ \sum_{K \in \mathcal{W}_I} \mathbf{1}_{CK} \lesssim \mathbf{1}_I .$$
\end{enumerate}
Here, $J \Subset I$ means that $J \subset I$ and $2^r \ell(J) \leq \ell(I)$; in words, $J$ is strongly contained in $I$.

\subsection{The Pivotal Condition}
The $pivotal \ constant$ $\mathscr{P}$ is the smallest constant in the following inequality. For any cube $I^0$, and any partition of $I^0$ into dyadic cubes $\{I_j; j \in \N \}$, there holds
\begin{equation}
\sum_{j \in \N} \sum_{K \in \mathcal{W}_{I_j}} \mathcal{P}_\alpha(K,\mathbf{1}_{I^0}\sigma)^2 w(K)
\leq \mathscr{P}^2 \sigma(I^0),
\end{equation}
where Poisson term $$ \mathcal{P}_\alpha(I,f)=\int_{\Rn} \frac{\ell(I)^\alpha}{(\ell(I)+ \dist(x,I))^{n+\alpha}} f(x) dx.$$

To obtain the best constants, we give the following.
\subsection{Proposition}\label{best constant}
The $A_2$ condition $(\ref{A2 condition})$ and testing condition $(\ref{testing condition-g^star})$ imply the finiteness of the pivotal constant $\mathscr{P}$. In particular, there holds $\mathscr{P} \lesssim \mathscr{A} + \mathscr{B}$.

\begin{proof}
We follow the strategy used in \cite{LL}. Taking the large enough constant $C$ in Proposition~\ref{overlap} such that $\frac{\alpha}{2} \geq n (\frac{2}{C-1})^2$. The $A_2$ condition and Proposition $\ref{overlap}$ give that
$$ \sum_{j \in \N} \sum_{K \in \mathcal{W}_{I_j}} \mathcal{P}_\alpha(K,\mathbf{1}_{CK}\sigma)^2 w(K)
\lesssim \mathscr{A}^2 \sum_{j \in \N} \sum_{K \in \mathcal{W}_{I_j}} \sigma(CK)
\lesssim \mathscr{A}^2 \sigma(I^0).$$
Thus, it is enough to treat the Poisson terms $\mathcal{P}_\alpha(K,\mathbf{1}_{I^0 \setminus CK} \sigma)$.

It is easy to verify
$$\mathcal{P}_\alpha(K,\mathbf{1}_{I^0 \setminus CK} \sigma)
\lesssim t \ \partial_t P^\alpha_t(\mathbf{1}_{I^0 \setminus CK} \sigma)(y,t),
\ \ \text{for \ any} \ y \in K, \ t \simeq \ell(K).$$
Therefore,
$$ \mathcal{P}_\alpha(K,\mathbf{1}_{I^0 \setminus CK} \sigma)^2 w(K)
\lesssim \int_{K} \iint_{W_K} \Big(\frac{t}{t+|x-y|}\Big)^{n\lambda}|\nabla P^\alpha_t(\mathbf{1}_{I^0 \setminus CK} \sigma)(y,t)|^2
\frac{w dx dt}{t^{n-1}} dy .$$
Since we have
\begin{eqnarray*}\aligned
{}&\sum_{j \in \N} \sum_{K \in \mathcal{W}_{I_j}}\int_{K} \iint_{W_K} \Big(\frac{t}{t+|x-y|}\Big)^{n\lambda}
|\nabla P^\alpha_t(\mathbf{1}_{I^0}\sigma)(y,t)|^2  \frac{w dx dt}{t^{n-1}} dy\\
&\le\sum_{j \in \N} \sum_{K \in \mathcal{W}_{I_j}} \int_{\Rn} \iint_{W_K} \Big(\frac{t}{t+|x-y|}\Big)^{n\lambda}|\nabla P^\alpha_t(\mathbf{1}_{I^0}\sigma)(y,t)|^2  \frac{w dx dt}{t^{n-1}} dy\\
&\le\int_{\Rn} \iint_{\widehat{I^0}} \Big(\frac{t}{t+|x-y|}\Big)^{n\lambda}|\nabla P^\alpha_t(\mathbf{1}_{I^0}\sigma)(y,t)|^2  \frac{w dx dt}{t^{n-1}} dy\\
&\le\mathscr{B}^2 \sigma(I^0),\endaligned
\end{eqnarray*}
and
\begin{eqnarray*}\aligned
{}&\sum_{j \in \N}\sum_{K \in \mathcal{W}_{I_j}}\int_{K} \iint_{W_K} \Big(\frac{t}{t+|x-y|}\Big)^{n\lambda}|\nabla P^\alpha_t(\mathbf{1}_{CK} \sigma)(y,t)|^2  \frac{w dx dt}{t^{n-1}} dy\\
&\le\sum_{j \in \N}\sum_{K \in \mathcal{W}_{I_j}}\int_{\Rn} \iint_{\widehat{CK}} \Big(\frac{t}{t+|x-y|}\Big)^{n\lambda}|\nabla P^\alpha_t(\mathbf{1}_{CK} \sigma)(y,t)|^2  \frac{w dx dt}{t^{n-1}} dy\\
&\le \mathscr{B}^2 \sum_{j \in \N}\sum_{K \in \mathcal{W}_{I_j}} \sigma(CK)\\
&\lesssim\mathscr{B}^2 \sigma(I^0),\endaligned
\end{eqnarray*}
the desired estimate follows immediately.

\end{proof}
\section{The Probabilistic Reduction}\label{Sec-Reduction}
Our next task is to simplify the proof of sufficiency. The probabilistic techniques we will use are taken from \cite{H}. We here need some fundamental tools, including the random dyadic grids, the probabilistic good/bad decompositions and the martingale difference expansions, which can be found in \cite{H},\cite{LL}, \cite{LM}, and essentially goes back to \cite{NTV}.
\subsection{The Generalized Result}
In order to prove the main theorem, it is enough to show the following generalized result.
\begin{equation}
\big\| g_{\psi,\lambda}^*(f \cdot \sigma) \big\|_{L^2(w)} \lesssim (\mathscr{A} + \mathscr{B}) \big\| f \big\|_{L^2(\sigma)},
\end{equation}
where
$$ g_{\psi,\lambda}^*(f)(x)=\bigg(\iint_{\R^{n+1}_{+}} \Big(\frac{t}{t+|x-y|}\Big)^{n\lambda}
|\psi_t*f(y)|^2 \frac{dy dt}{t^{n+1}}\bigg)^{1/2},$$
$\psi_t(x)=\frac{1}{t^n} \psi(\frac{x}{t})$ and $\psi$ satisfies the following conditions:
\begin{enumerate}
\item [(1)] $|\psi(x)| \lesssim (1+|x|)^{-n-\alpha}$;
\item [(2)] $|\psi(x)-\psi(y)| \lesssim |x-y|^\alpha (1+|x|)^{-n - \alpha}$.
\end{enumerate}
A particular case of the above function class was introduced by Wilson \cite[p.~114]{W}. However, we do not need the cancellation property of $\psi$ in this paper.

\subsection{Averaging over Good Whitney Regions}
Let $f \in L^2(\sigma)$.
Note that the position and goodness of $R + \beta$ are independent (see \cite{H}). Therefore, one can write
\begin{align*}
& \big\| g_{\psi,\lambda}^*(f \cdot \sigma)\big\|_{L^2(w)}^2 \\
&= \int_{\Rn} \iint_{\R^{n+1}_{+}} \Big(\frac{t}{t+|y|}\Big)^{n\lambda}
|\psi_t*(f \cdot \sigma)(x-y)|^2 \frac{dy dt}{t^{n+1}} w dx \\
&= \iint_{\R^{n+1}_{+}} \int_{\Rn} |\psi_t*(f \cdot \sigma)(x-y)|^2 \Big(\frac{t}{t+|y|}\Big)^{n\lambda}
 \frac{dy}{t^n} w dx \frac{dt}{t}\\
&= \mathbb{E}_{\beta} \sum_{R \in \mathcal{D}_0} \iint_{W_{R + \beta}} \int_{\Rn} |\psi_t*(f \cdot \sigma)(x-y)|^2 \Big(\frac{t}{t+|y|}\Big)^{n\lambda} \frac{dy}{t^n} w dx \frac{dt}{t}\\
&= \frac{1}{\pi_{good}} \sum_{R \in \mathcal{D}_0} \mathbb{E}_{\beta}(\mathbf{1}_{good}(R+\beta)) \mathbb{E}_{\beta} \iint_{W_{R + \beta}} \int_{\Rn} |\psi_t*(f \cdot \sigma)(x-y)|^2 \Big(\frac{t}{t+|y|}\Big)^{n\lambda} \frac{dy}{t^n} w dx \frac{dt}{t}\\
&= \frac{1}{\pi_{good}} \sum_{R \in \mathcal{D}_0} \mathbb{E}_{\beta}\bigg(\mathbf{1}_{good}(R+\beta) \iint_{W_{R + \beta}} \int_{\Rn} |\psi_t*(f \cdot \sigma)(x-y)|^2 \Big(\frac{t}{t+|y|}\Big)^{n\lambda} \frac{dy}{t^n} w dx \frac{dt}{t} \bigg)\\
&= \frac{1}{\pi_{good}} \mathbb{E}_{\beta} \sum_{R \in \mathcal{D}_{good}} \iint_{W_R} \int_{\Rn} |\psi_t*(f \cdot \sigma)(x-y)|^2 \Big(\frac{t}{t+|y|}\Big)^{n\lambda} \frac{dy}{t^n} w dx \frac{dt}{t}.
\end{align*}

With the monotone convergence theorem, it suffices to show that there exists a constant $\mathcal{C} > 0$ such that
for any $s \in \N$, we have
$$ \sum_{\substack{R \in \mathcal{D}_{good} \\ R\subset [-2^s, 2^s]^n \\ 2^{-s}\le \ell(R)\leq 2^s}} \iint_{W_R} \int_{\Rn} |\psi_t*(f \cdot \sigma)(x-y)|^2 \Big(\frac{t}{t+|y|}\Big)^{n\lambda} \frac{dy}{t^n} w dx \frac{dt}{t}
\leq \mathcal{C} (\mathscr{A} + \mathscr{B})^2 \big\| f \big\|_{L^2(\sigma)}^2.$$

\subsection{The Final Reduction}
In order to get the further reduction, we introduce the martingale difference decomposition. Define
$$\mathbb{E}_Q^{\sigma}f := \frac{1}{\sigma(Q)} \int_Q f d\sigma,$$
assuming that $\sigma(Q) > 0$, otherwise set it to be zero. For the martingale differences,
$$ \Delta_Q^{\sigma} f := \sum_{Q' \in ch(Q)} (\mathbb{E}_{Q'}^{\sigma}f - \mathbb{E}_Q^{\sigma}f)\mathbf{1}_{Q'}.$$
For fixed $s \in \N$, by Lebesgue differentiation theorem, we can write
$$ f = \sum_{\substack{Q \in \mathcal{D}\\ \ell(Q) \leq 2^s}} \Delta_Q^{\sigma}f
+ \sum_{\substack{Q \in \mathcal{D}\\ \ell(Q)=2^s}} (\mathbb{E}_Q^{\sigma}f) \mathbf{1}_Q .$$
Since $\{ \Delta_Q^{\sigma}f \}_{Q \in \mathcal{D}}$ is a family of orthogonal functions, we have
$$ \big\|f\big\|_{L^2(\sigma)}^2 = \sum_{\substack{Q \in \mathcal{D}\\ \ell(Q) \leq 2^s}}
\big\| \Delta_Q^{\sigma}f \big\|_{L^2(\sigma)}^2
+ \sum_{\substack{Q \in \mathcal{D}\\ \ell(Q)=2^s}} \big\|(\mathbb{E}_Q^{\sigma}f) \mathbf{1}_Q \big\|_{L^2(\sigma)}^2 .$$

Now we claim that we can assume that $f$ is compactly supported, say $\supp f\subset Q^0$. Let $\mathscr F$ denote the subspace of $L^2(\sigma)$ which has compact support. We shall show that
\begin{equation}\label{eq:compact}
\mathscr K:=\sup_{\substack{f\in \mathscr F\\ \|f\|_{L^2(\sigma)}=1}} \|g_{\psi,\lambda}^*(f\sigma)\|_{L^2(w)} <\infty.
\end{equation}
Indeed, if \eqref{eq:compact} is proved, then for any $f\in L^2(\sigma)$ and $\varepsilon>0$, there exists some cube $Q$ such that
\[
\|f-f\chi^{}_{Q}\|_{L^2(\sigma)}<\varepsilon \|f\|_{L^2(\sigma)},
\]
For simplicity, set $g:=f-f\chi^{}_{Q}$. Then we have
\begin{align*}
&\sum_{\substack{R \in \mathcal{D}_{good} \\ R \subset [-2^s, 2^s]^n \\ 2^{-s}\le \ell(R)\leq 2^s}} \iint_{W_R} \int_{\Rn} |\psi_t*(f \cdot \sigma)(x-y)|^2 \Big(\frac{t}{t+|y|}\Big)^{n\lambda} \frac{dy}{t^n} w dx \frac{dt}{t}\\
&\leq 2 \sum_{\substack{R \in \mathcal{D}_{good} \\ R \subset [-2^s, 2^s]^n\\ 2^{-s}\le \ell(R)\leq 2^s}} \iint_{W_R} \int_{\Rn} |\psi_t*(f\chi^{}_{Q} \cdot \sigma)(x-y)|^2 \Big(\frac{t}{t+|y|}\Big)^{n\lambda} \frac{dy}{t^n} w dx \frac{dt}{t}\\
&\quad + 2 \sum_{\substack{R \in \mathcal{D}_{good} \\ R \subset [-2^s, 2^s]^n \\ 2^{-s}\le \ell(R)\leq 2^s}} \iint_{W_R} \int_{\Rn} |\psi_t*(g \cdot \sigma)(x-y)|^2 \Big(\frac{t}{t+|y|}\Big)^{n\lambda} \frac{dy}{t^n} w dx \frac{dt}{t}.
\end{align*}
Substitute with
\[ f(x)=\frac {\ell(Q)^\alpha}{(\ell(Q)+\textup{dist}(x,Q))^{n+\alpha}}\chi^{}_{Q'\setminus 4\sqrt n Q} \]
in \eqref{eq:compact} and using similar arguments as that in \cite{LL}, we get
\[ \int_{Q'} \frac{\ell(Q)^{2 \alpha}}{(\ell(Q)+\dist(z,Q))^{2(n+\alpha)}}d\sigma(z) w(Q)\lesssim \mathscr{K}^2+ \mathscr{A}^2 .\]
Then by letting $Q'$ increase to $\mathbb R^n$, we know that \eqref{eq:compact} and the $A_2$ condition imply the Poisson type $A_2$ condition. Therefore,
\begin{eqnarray*}
&&\iint_{W_R} \int_{\Rn} |\psi_t*(g \cdot \sigma)(x-y)|^2 \Big(\frac{t}{t+|y|}\Big)^{n\lambda} \frac{dy}{t^n} w dx \frac{dt}{t}\\
&\le& \iint_{W_R}  \int_{\Rn} \|g\|_{L^2(\sigma)}^2 \|\psi_t(x-y-\cdot)\|_{L^2(\sigma)}^2 \Big(\frac{t}{t+|y|}\Big)^{n\lambda} \frac{dy}{t^n} w dx \frac{dt}{t}\\
&\le& C_n\int_{\Rn} \frac{\ell(R)^{2\alpha}}{(\ell(R)+\dist(z,R))^{2(n+\alpha)}}d\sigma(z) w(R) \|g\|_{L^2(\sigma)}^2\\
&\le& C_n(\mathscr K^2+\mathscr{A}^2)\varepsilon^2 \|f\|_{L^2(\sigma)}^2
\end{eqnarray*}
Then by taking sufficiently large cube $Q$ such that  $2^{(2s+2)n}C_n(\mathscr K^2+\mathscr{A}^2)\varepsilon^2< \mathscr K^2$. We finally get
\begin{eqnarray*}
\sum_{\substack{R \in \mathcal{D}_{good} \\ R\subset [-2^s, 2^s]^n \\ 2^{-s}\le \ell(R)\leq 2^s}} \iint_{W_R} \int_{\Rn} |\psi_t*(f \cdot \sigma)(x-y)|^2 \Big(\frac{t}{t+|y|}\Big)^{n\lambda} \frac{dy}{t^n} w dx \frac{dt}{t}
&\le&4 \mathscr K^2 \|f\|_{L^2(\sigma)},
\end{eqnarray*}
which means that we reduce the problem to prove \eqref{eq:compact}.
Then by repeating the previous arguments, we further reduce the problem to estimate
\[
\sum_{\substack{R \in \mathcal{D}_{good} \\ R\subset [-2^s, 2^s]^n \\ 2^{-s}\le \ell(R)\leq 2^s}} \iint_{W_R} \int_{\Rn} |\psi_t*(f \cdot \sigma)(x-y)|^2 \Big(\frac{t}{t+|y|}\Big)^{n\lambda} \frac{dy}{t^n} w dx \frac{dt}{t},
\]
where $f$ has compact support.  Assume that $\supp f\subset [-2^{s'}, 2^{s'}]^n$. Without loss of generality, we can assume that $s\ge s'+1$. Then it suffices to estimate
\[
\sum_{\substack{R \in \mathcal{D}_{good} \\ R\subset [-2^s, 2^s]^n \\ 2^{-s}\le \ell(R)\leq 2^s}} \iint_{W_R} \int_{\Rn} |\psi_t*( f\mathbf 1_{[-2^{s-1},2^{s-1}]^n} \cdot \sigma)(x-y)|^2 \Big(\frac{t}{t+|y|}\Big)^{n\lambda} \frac{dy}{t^n} w dx \frac{dt}{t}.
\]
Denote $\mathscr F_s$ the subspace of $\mathscr F$ which supported in $[-2^{s-1},2^{s-1}]^n$.
\[
\mathscr K_s:=\sup_{\substack{f \in \mathscr{F}_s\\ \|f\|_{L^2(\sigma)=1}}}\sum_{\substack{R \in \mathcal{D}_{good}\\ R\subset [-2^s, 2^s]^n \\ 2^{-s}\le \ell(R)\leq 2^s}} \iint_{W_R} \int_{\Rn} |\psi_t*(f \cdot \sigma)(x-y)|^2 \Big(\frac{t}{t+|y|}\Big)^{n\lambda} \frac{dy}{t^n} w dx \frac{dt}{t}
\]
Similar arguments as the previous show that
\begin{eqnarray*}
&& \iint_{W_R} \int_{\Rn} |\psi_t*(f  \cdot \sigma)(x-y)|^2 \Big(\frac{t}{t+|y|}\Big)^{n\lambda} \frac{dy}{t^n} w dx \frac{dt}{t}\\
&\le& C_n\int_{[-2^{s-2},2^{s-2}]^n} \frac{\ell(R)^{2\alpha}}{(\ell(R)+\dist(z,R))^{2(n+\alpha)}}d\sigma(z) w(R) \|f\|_{L^2(\sigma)}^2\\
&\le& 2^{2sn}C_n \mathscr{A}^2 \|f\|_{L^2(\sigma)}^2,
\end{eqnarray*}
which means that $\mathscr K_s\le 2^{(4s+2)n} C_n \mathscr{A}^2<\infty $. Using the martingale decomposition, we can write
\[
f =\sum_{\substack{Q\in\mathcal D\\ \ell(Q)\le 2^s}}\Delta_Q^\sigma f,
\]
when $\ell(Q)=2^s$, $\Delta_Q^\sigma$ should be understood as $\Delta_Q^\sigma + \mathbb{E}_Q^\sigma$.
Denote
\[
f_{\textup{good}}=\sum_{\substack{Q\in \mathcal D_{\textup{good}}\\ \ell(Q)\le 2^s}}\Delta_Q^\sigma f.
\]
Again, we can set $\tilde g:= f -f_{\textup{good}}$. For any $\varepsilon>0$, choosing $r$ sufficiently large such that $\|\tilde g\|_{L^2(\sigma)}<\varepsilon$, see \cite{Lacey1}. Then we have
\begin{align*}
\mathscr{K}_s &\leq
2\sup_{\substack{f\in\mathscr F\\ \|f\|_{L^2(\sigma)=1}}}\sum_{\substack{R \in \mathcal{D}_{good}\\ R\subset [-2^s, 2^s]^n \\ 2^{-s}\le \ell(R)\leq 2^s}} \iint_{W_R} \int_{\Rn} |\psi_t*(f_{\textup{good}} \cdot \sigma)(x-y)|^2 \Big(\frac{t}{t+|y|}\Big)^{n\lambda} \frac{dy}{t^n} w dx \frac{dt}{t}\\
&\quad + 2 \mathscr{K}_s \|\tilde g\|_{L^2(\sigma)}^2.
\end{align*}
By taking $\varepsilon=1/2$, (which means that $r$ is independent of $s$) we reduce the problem to prove
\begin{equation}
\sum_{\substack{R \in \mathcal{D}_{good} \\ \ell(R) \leq 2^s}} \iint_{W_R} \int_{\Rn} \Big|\sum_{\substack{Q \in \mathcal{D}_{good} \\ \ell(Q) \leq 2^s}} \psi_t*(\Delta_Q^{\sigma}f \cdot \sigma)(x-y)\Big|^2  \Big(\frac{t}{t+|y|}\Big)^{n\lambda} \frac{dy}{t^{n}} w dx \frac{dt}{t}
\lesssim (\mathscr{A} + \mathscr{B})^2 \big\|f\big\|_{L^2(\sigma)}^2.
\end{equation}

\section{Some Lemmas And Elementary Estimates}\label{Sec-Lemmas}
To prove the boundedness of $g_{\lambda}^*(\cdot \sigma)$ from $L^2(\sigma)$ to $L^2(w)$, we here present some crucial estimates and lemmas.
\subsection{Elementary Estimate 1}
Let $0 < \alpha \leq n(\lambda - 2)/2$. For given cubes $Q, R \in \mathcal{D}$ and $(x,t) \in W_R$,
we have the following estimate
\begin{equation}\label{Key Estimate-1}
\bigg( \int_{\Rn} |\psi_t*((\Delta_Q^{\sigma}f) \sigma)(x-y)|^2 \Big(\frac{t}{t+|y|}\Big)^{n\lambda}\frac{dy}{t^n}\bigg)^{1/2}
\lesssim \frac{\ell(R)^\alpha \ \sigma(Q)^{1/2}}{(\ell(R)+d(Q,R))^{n+\alpha}} \big\| \Delta_Q^{\sigma}f \big\|_{L^2(\sigma)}.
\end{equation}

\begin{proof}
By the size condition, we obtain
$$|\psi_t*((\Delta_Q^{\sigma}f )\cdot \sigma)(x-y)|
\lesssim \int_{\Rn}\frac{t^\alpha}{(t+|x-y-z|)^{n+\alpha}} |\Delta_Q^{\sigma}f(z)|d\sigma(z). $$
Since $z \in Q$ and $x \in R$, $|x-z|\geq d(Q,R)$.

If $|y|\leq \frac12 d(Q,R)$, then $|x-y-z| \geq |x-z|-|y| \geq \frac12 d(Q,R)$.
Thus,
\begin{align*}
|\psi_t*((\Delta_Q^{\sigma}f) \sigma)(x-y)|
&\lesssim \frac{\ell(R)^\alpha}{(\ell(R)+d(Q,R))^{n+\alpha}} \big\| \Delta_Q^{\sigma}f \big\|_{L^1(\sigma)} \\
&\leq \frac{\ell(R)^\alpha}{(\ell(R)+d(Q,R))^{n+\alpha}} \sigma(Q)^{1/2} \big\| \Delta_Q^{\sigma}f \big\|_{L^2(\sigma)},
\end{align*}
and
\begin{align*}
\bigg( \int_{|y|\leq \frac12 d(Q,R)} |\psi_t*((\Delta_Q^{\sigma}f) \sigma)(x-y)|^2 \Big(\frac{t}{t+|y|}\Big)^{n\lambda}\frac{dy}{t^n}\bigg)^{1/2}
\lesssim \frac{\ell(R)^\alpha \ \sigma(Q)^{1/2}}{(\ell(R)+d(Q,R))^{n+\alpha}} \big\| \Delta_Q^{\sigma}f \big\|_{L^2(\sigma)}.
\end{align*}

If $|y| > \frac12 d(Q,R)$, then
$$ \Big(\frac{t}{t+|y|}\Big)^{n\lambda}\frac{1}{t^n} \lesssim \frac{\ell(R)^{n\lambda - n}}{(\ell(R)+d(Q,R))^{n\lambda}}.$$
Hence, by Young's inequality, it yields that
\begin{align*}
&\bigg( \int_{|y| > \frac12 d(Q,R)} |\psi_t*((\Delta_Q^{\sigma}f) \sigma)(x-y)|^2 \Big(\frac{t}{t+|y|}\Big)^{n\lambda}\frac{dy}{t^n}\bigg)^{1/2} \\
&\lesssim \frac{\ell(R)^{\frac{n\lambda}{2} - \frac{n}{2}}}{(\ell(R)+d(Q,R))^{\frac{n\lambda}{2}}} \big\|\psi_t*((\Delta_Q^{\sigma}f )\sigma)(x-\cdot)\big\|_{L^2(\Rn)} \\
&\leq \frac{\ell(R)^{\frac{n\lambda}{2} - \frac{n}{2}}}{(\ell(R)+d(Q,R))^{\frac{n\lambda}{2}}} \big\|\psi_t\big\|_{L^2(\Rn)} \big\|\Delta_Q^{\sigma}f \big\|_{L^1(\sigma)} \\
&\lesssim \frac{\ell(R)^{\frac{n\lambda}{2} - \frac{n}{2}} t^{-\frac{n}{2}}}{(\ell(R)+d(Q,R))^{\frac{n\lambda}{2}}}
\sigma(Q)^{1/2} \big\| \Delta_Q^{\sigma}f \big\|_{L^2(\sigma)} \\
&\lesssim \frac{\ell(R)^{\frac{n\lambda}{2} - n}}{(\ell(R)+d(Q,R))^{\frac{n\lambda}{2}}} \sigma(Q)^{1/2}
\big\| \Delta_Q^{\sigma}f \big\|_{L^2(\sigma)} \\
&\leq \frac{\ell(R)^\alpha}{(\ell(R)+d(Q,R))^{n+\alpha}} \sigma(Q)^{1/2} \big\| \Delta_Q^{\sigma}f \big\|_{L^2(\sigma)},
\end{align*}
where we have used the condition $0 < \alpha \leq n(\lambda - 2)/2$ in the last step.

This completes the proof of $(\ref{Key Estimate-1})$.
\end{proof}

\subsection{Elementary Estimate 2}
Let $0 < \alpha \leq n(\lambda - 2)/2$. Assume that $Q, R \in \mathcal{D}$ are given cubes with $\ell(Q) < \ell(R)$,
$\ell(Q) < 2^s$ and $(x,t) \in W_R$. Then we have the following estimate
\begin{equation}\label{Key Estimate-2}
\bigg( \int_{\Rn} |\psi_t*(\Delta_Q^{\sigma}f \cdot \sigma)(x-y)|^2 \Big(\frac{t}{t+|y|}\Big)^{n\lambda}\frac{dy}{t^n}\bigg)^{1/2}
\lesssim \frac{\ell(Q)^{\alpha/2} \ell(R)^{\alpha/2} \ \sigma(Q)^{1/2}}{(\ell(R)+d(Q,R))^{n+\alpha}}  \big\| \Delta_Q^{\sigma}f \big\|_{L^2(\sigma)}.
\end{equation}

\begin{proof}
Let $z_Q$ be the center of $Q$. By the cancellation condition $\int_Q \Delta_Q^{\sigma}f \sigma dx =0$,we have
$$\psi_t*(\Delta_Q^{\sigma}f \cdot \sigma)(x-y)
= \int_{Q}(\psi_t(x-y-z)-\psi_t(x-y-z_Q)) \Delta_Q^{\sigma}f(z) d\sigma(z). $$
Since $|z-z_Q| \leq \frac{\sqrt{n}}{2}\ell(Q) \leq \frac{\sqrt{n}}{4} \ell(R) < \frac{\sqrt{n}}{2} t \leq \frac{\sqrt{n}}{2} \ell(R)$, we have
\begin{align*}
|\psi_t(x-y-z)-\psi_t(x-y-z_Q)|
\lesssim \frac{|z-z_Q|^\alpha}{(t+|x-y-z|)^{n+\alpha}}
\lesssim \frac{\ell(Q)^{\alpha/2} \ell(R)^{\alpha/2}}{(t+|x-y-z|)^{n+\alpha}}.
\end{align*}
Making use the similar arguments as in the preceding subsection, we will obtain the inequality $(\ref{Key Estimate-2})$.
\end{proof}
\subsection{Some Lemmas}
For convenience, we here present two key lemmas, which will be used later.
\begin{lemma}[\cite{LL}]\label{A-alpha-QR}
Let $$A_{QR}^{\epsilon}=\frac{\ell(Q)^{\epsilon/2} \ell(R)^{\epsilon/2}}{D(Q,R)^{n+\epsilon}} \sigma(Q)^{1/2} w(R)^{1/2},$$
where $Q,R \in \mathcal{D}$ and $\epsilon > 0$. Then for any $x_Q, y_R \geq 0$, we have the following estimate
$$
\Big( \sum_{Q,R \in \mathcal{D}} A_{QR}^{\epsilon} x_Q y_R \Big)^2
\lesssim \mathscr{A}^2 \sum_{Q \in \mathcal{D}} x_Q^2 \times \sum_{R \in \mathcal{D}} y_R^2.
$$
In particular, there holds that
$$
\sum_{R \in \mathcal{D}} \Big( \sum_{Q \in \mathcal{D}} A_{QR}^\epsilon x_{Q} \Big)^2
\lesssim \mathscr{A}^2 \sum_{Q \in \mathcal{D}} x_{Q}^2.
$$
\end{lemma}

\begin{lemma}\label{lemma-R-K-S}
Let $0 < \alpha \leq n(\lambda - 2)/2$. Given three cubes $R \subset K \subset S$, and function $f$ satisfies $\supp(f) \cap S = \emptyset$. If $\dist(R,\partial K) \geq \ell(R)^\gamma \ell(K)^{1-\gamma}$, then there holds
\begin{equation}\label{R-K-S}
\int_{\Rn} \iint_{W_R} |\psi_t*(f \cdot \sigma)(x-y)|^2 w dx \Big(\frac{t}{t+|y|}\Big)^{n\lambda}
\frac{dt dy}{t^{n+1}} \lesssim \bigg(\frac{\ell(R)}{\ell(K)}\bigg)^\alpha \mathcal{P}_\alpha(K,|f|\sigma)^2 w(R).
\end{equation}
\end{lemma}

\begin{proof}
First, we shall prove, for any $z \not\in S$,
\begin{equation}\label{d(z,R)-d(z,K)}
\frac{\ell(R)^\alpha}{(\ell(R) + \dist(z,R))^{n+\alpha}}
\leq \Big[ \frac{\ell(R)}{\ell(K)} \Big]^{\alpha /2} \frac{\ell(K)^\alpha}{(\ell(K)+ \dist(z,K))^{n+\alpha}}.
\end{equation}
In fact, since $\dist(z,R) \geq \dist(z,K) + \dist(R,\partial K)$, we have
\begin{align*}
\frac{\ell(R)^\alpha}{(\ell(R)+ \dist(z,R))^{n+\alpha}}
&= \bigg(\frac{\ell(R)}{\ell(K)}\bigg)^\alpha \frac{\ell(K)^\alpha}{(\ell(R)+ \dist(z,R))^{n+\alpha}} \\
&\lesssim \bigg( \frac{\ell(R)}{\ell(K)} \bigg)^{\alpha -(n+\alpha)\gamma} \frac{\ell(K)^\alpha}{(\ell(K)+ \dist(z,K))^{n+\alpha}}.
\end{align*}

Secondly, we turn to the estimate of $(\ref{R-K-S})$. Decompose
\begin{align*}
&\int_{\Rn} |\psi_t*(f \cdot \sigma)(x-y)|^2 \Big(\frac{t}{t+|y|}\Big)^{n\lambda} \frac{dy}{t^{n}} \\
&\leq \int_{\Rn} \bigg(\int_{z:|y| \leq \frac12 \dist(z,R)}|\psi_t(x-y-z)| |f(z)| d\sigma(z) \bigg)^2 \Big(\frac{t}{t+|y|}\Big)^{n\lambda} \frac{dy}{t^{n}}\\
&\quad + \int_{\Rn} \bigg(\int_{z:|y| > \frac12 \dist(z,R)}|\psi_t(x-y-z)| |f(z)| d\sigma(z) \bigg)^2 \Big(\frac{t}{t+|y|}\Big)^{n\lambda} \frac{dy}{t^{n}} \\
&:= \mathcal{E}_1 + \mathcal{E}_2.
\end{align*}

For $(x,t) \in W_R$, and $z \not\in S$, we have
$$ |\psi_t(x-y-z)| \lesssim \frac{t^\alpha}{(t+|x-y-z|)^{n+\alpha}}
\lesssim \frac{\ell(R)^\alpha}{(\ell(R)+|x-y-z|)^{n+\alpha}}.$$

If $|y| \leq \frac12 \dist(z,R)$, $|x-y-z| \geq |x-z|-|y| \geq \frac12 \dist(z,R)$. Then by $(\ref{d(z,R)-d(z,K)})$
\begin{align*}
|\psi_t(x-y-z)|
\lesssim \frac{\ell(R)^\alpha}{(\ell(R)+ \dist(z,R))^{n+\alpha}}
\lesssim \bigg( \frac{\ell(R)}{\ell(K)} \bigg)^{\alpha/2} \frac{\ell(K)^\alpha}{(\ell(K)+ \dist(z,K))^{n+\alpha}}.
\end{align*}
Hence, there holds
\begin{align*}
\mathcal{E}_1
&\lesssim \bigg(\frac{\ell(R)}{\ell(K)}\bigg)^\alpha \int_{\Rn} \bigg(\int_{\Rn} \frac{\ell(K)^\alpha}{(\ell(K)+ \dist(z,K))^{n+\alpha}} |f(z)| d\sigma(z) \bigg)^2 \Big(\frac{t}{t+|y|}\Big)^{n\lambda} \frac{dy}{t^{n}} \\
&\lesssim \bigg(\frac{\ell(R)}{\ell(K)}\bigg)^\alpha \mathcal{P}_\alpha(K,|f|\sigma)^2 .
\end{align*}

If $|y| > \frac12 dist(z,R)$, the inequality $(\ref{d(z,R)-d(z,K)})$ and Young's inequality imply that
\begin{align*}
\mathcal{E}_2
&\lesssim t^n \int_{\Rn} \bigg(\int_{z:|y| > \frac12 \dist(z,R)}|\psi_t(x-y-z)| \frac{\ell(R)^{\frac{n\lambda}{2}-n}}
{(\ell(R)+ \dist(z,R))^{\frac{n\lambda}{2}}} |f(z)| d\sigma(z) \bigg)^2 dy \\
&\leq t^n \int_{\Rn} \bigg(\int_{\Rn}|\psi_t(x-y-z)| \frac{\ell(R)^\alpha}{(\ell(R)+ \dist(z,R))^{n+\alpha}} |f(z)| d\sigma(z)\bigg)^2 dy \\
&\lesssim t^n \bigg(\frac{\ell(R)}{\ell(K)}\bigg)^\alpha \int_{\Rn} \bigg(\int_{\Rn}|\psi_t(x-y-z)| \frac{\ell(K)^\alpha}{(\ell(K)+ \dist(z,K))^{n+\alpha}} |f(z)| d\sigma(z)\bigg)^2 dy \\
&\leq t^n \bigg(\frac{\ell(R)}{\ell(K)}\bigg)^\alpha \big\| \psi_t \big\|_{L^2(\Rn)}^2 \bigg(\int_{\Rn} \frac{\ell(K)^\alpha}{(\ell(K)+ \dist(z,K))^{n+\alpha}} |f(z)| d\sigma(z) \bigg)^2 \\
&\lesssim \bigg(\frac{\ell(R)}{\ell(K)}\bigg)^\alpha \mathcal{P}_\alpha(K,|f|\sigma)^2.
\end{align*}
Consequently, the inequality $(\ref{R-K-S})$ is concluded from the above estimates.
\end{proof}
\section{The Sufficiency in The Main Theorem}\label{Sec-Sufficiency}
In this section, we undertake to prove the sufficiency. We shall divide the collection
$\{Q; Q \in \mathcal{D}_{good}, \ell(Q) \leq 2^s\}$ into the following four parts. The last one is the core and quite complicated.
\subsection{The Case $\ell(Q) < \ell(R)$.}
In this case, we must have $\ell(Q) < 2^s$. It follows from $(\ref{Key Estimate-2})$ and Lemma $\ref{A-alpha-QR}$ that
\begin{align*}
&\sum_{\substack{R \in \mathcal{D}_{good} \\ \ell(R) \leq 2^s}} \iint_{W_R} \int_{\Rn} \Big|\sum_{\substack{Q \in \mathcal{D}_{good} \\ \ell(Q) < \ell(R)}} \psi_t*(\Delta_Q^{\sigma}f \cdot \sigma)(x-y)\Big|^2  \Big(\frac{t}{t+|y|}\Big)^{n\lambda} \frac{dy}{t^{n}} w dx \frac{dt}{t} \\
&\leq \sum_{\substack{R \in \mathcal{D}_{good} \\ \ell(R) \leq 2^s}} \iint_{W_R}
\bigg[\sum_{\substack{Q \in \mathcal{D}_{good} \\ \ell(Q) < \ell(R)}} \bigg(\int_{\Rn} |\psi_t*(\Delta_Q^{\sigma}f \cdot \sigma)(x-y)|^2 \Big(\frac{t}{t+|y|}\Big)^{n\lambda}\frac{dy}{t^n}\bigg)^{1/2}\bigg]^2 w dx \frac{dt}{t} \\
&\lesssim \sum_{\substack{R \in \mathcal{D}_{good} \\ \ell(R) \leq 2^s}} \iint_{W_R}
\bigg[\sum_{\substack{Q \in \mathcal{D}_{good} \\ \ell(Q) < \ell(R)}} \frac{\ell(Q)^{\alpha/2}\ell(R)^{\alpha/2}}
{(\ell(R)+d(Q,R))^{n+\alpha}} \sigma(Q)^{1/2} \big\| \Delta_Q^{\sigma}f \big\|_{L^2(\sigma)}\bigg]^2 w dx \frac{dt}{t} \\
&\lesssim \sum_{R \in \mathcal{D}_{good}}
\bigg(\sum_{Q \in \mathcal{D}_{good}} A_{QR}^\alpha \big\| \Delta_Q^{\sigma}f \big\|_{L^2(\sigma)}\bigg)^2
\lesssim {\mathscr{A}^2} \big\| f \big\|_{L^2(\sigma)}^2 .
\end{align*}
\subsection{The Case $\ell(Q) \geq \ell(R)$ and $d(Q,R) > \ell(R)^{\gamma} \ell(Q)^{1-\gamma}$.}
We claim that  there holds in this case
\begin{equation}\label{ell-d-D}
\frac{\ell(R)^\alpha}{(\ell(R)+d(Q,R))^{n+\alpha}}
\lesssim \frac{\ell(Q)^{\alpha/2} \ell(R)^{\alpha/2}}{D(Q,R)^{n+\alpha}}.
\end{equation}
Indeed, if $\ell(Q) \leq d(Q,R)$, it is obvious that
\begin{align*}
\frac{\ell(R)^\alpha}{(\ell(R)+d(Q,R))^{n+\alpha}}
\lesssim \frac{\ell(R)^\alpha}{D(Q,R)^{n+\alpha}}
\leq \frac{\ell(Q)^{\alpha/2} \ell(R)^{\alpha/2}}{D(Q,R)^{n+\alpha}}.
\end{align*}
If $\ell(Q) > d(Q,R)$, then $D(Q,R) \simeq \ell(Q)$. Using $d(Q,R) > \ell(R)^{\gamma} \ell(Q)^{1-\gamma}$ and
$\gamma = \frac{\alpha}{2(n+\alpha)}$, we obtain that
\begin{align*}
\frac{\ell(R)^\alpha}{(\ell(R)+d(Q,R))^{n+\alpha}}
\leq \frac{\ell(R)^\alpha}{d(Q,R)^{n+\alpha}}
\leq \frac{\ell(Q)^{\alpha/2} \ell(R)^{\alpha/2}}{\ell(Q)^{n+\alpha}}
\simeq \frac{\ell(Q)^{\alpha/2} \ell(R)^{\alpha/2}}{D(Q,R)^{n+\alpha}}.
\end{align*}

Then Lemma $\ref{A-alpha-QR}$ and the inequalities $(\ref{Key Estimate-1})$, $(\ref{ell-d-D})$ give that
\begin{align*}
&\sum_{\substack{R \in \mathcal{D}_{good} \\ \ell(R) \leq 2^s}} \iint_{W_R} \int_{\Rn} \Big|\sum_{\substack{Q \in \mathcal{D}_{good}:\ell(Q) \geq \ell(R) \\ d(Q,R) > \ell(R)^{\gamma} \ell(Q)^{1-\gamma}}} \psi_t*(\Delta_Q^{\sigma}f \cdot \sigma)(x-y)\Big|^2  \Big(\frac{t}{t+|y|}\Big)^{n\lambda} \frac{dy}{t^{n}} w dx \frac{dt}{t} \\
&\leq \sum_{\substack{R \in \mathcal{D}_{good} \\ \ell(R) \leq 2^s}} \iint_{W_R}
\bigg[\sum_{\substack{Q \in \mathcal{D}_{good}:\ell(Q) \geq \ell(R) \\ d(Q,R) > \ell(R)^{\gamma} \ell(Q)^{1-\gamma}}} \bigg(\int_{\Rn} |\psi_t*(\Delta_Q^{\sigma}f \cdot \sigma)(x-y)|^2 \Big(\frac{t}{t+|y|}\Big)^{n\lambda}\frac{dy}{t^n}\bigg)^{1/2}\bigg]^2 w dx \frac{dt}{t} \\
&\lesssim \sum_{\substack{R \in \mathcal{D}_{good} \\ \ell(R) \leq 2^s}} \iint_{W_R}
\bigg[\sum_{\substack{Q \in \mathcal{D}_{good}:\ell(Q) \geq \ell(R) \\ d(Q,R) > \ell(R)^{\gamma} \ell(Q)^{1-\gamma}}}  \frac{\ell(R)^\alpha}{(\ell(R)+d(Q,R))^{n+\alpha}} \sigma(Q)^{1/2} \big\| \Delta_Q^{\sigma}f \big\|_{L^2(\sigma)}\bigg]^2
w dx \frac{dt}{t} \\
&\lesssim \sum_{R \in \mathcal{D}_{good}}
\bigg(\sum_{Q \in \mathcal{D}_{good}} A_{QR}^{\alpha} \big\| \Delta_Q^{\sigma}f \big\|_{L^2(\sigma)}\bigg)^2
\lesssim {\mathscr{A}^2 }\big\| f \big\|_{L^2(\sigma)}^2 .
\end{align*}
\subsection{The Case $\ell(R) \leq \ell(Q) \leq 2^r \ell(R)$ and $d(Q,R) \leq \ell(R)^{\gamma} \ell(Q)^{1-\gamma}$.}
In this case, it is trivial that $D(Q,R) \simeq \ell(Q) \simeq \ell(R)$.
Thus
$$
\frac{\ell(R)^\alpha}{(\ell(R)+d(Q,R))^{n+\alpha}}
\leq \ell(R)^{-n} \simeq \frac{\ell(Q)^{\alpha/2} \ell(R)^{\alpha/2}}{D(Q,R)^{n+\alpha}} .
$$
Then proceeding as we did in the previous subsection, we obtain that
\begin{align*}
&\sum_{\substack{R \in \mathcal{D}_{good} \\ \ell(R) \leq 2^s}} \iint_{W_R} \int_{\Rn} \Big|\sum_{\substack{\ell(R) \leq \ell(Q) \leq 2^r \ell(R) \\ d(Q,R) \leq \ell(R)^{\gamma} \ell(Q)^{1-\gamma}}} \psi_t*(\Delta_Q^{\sigma}f \cdot \sigma)(x-y)\Big|^2 \Big(\frac{t}{t+|y|}\Big)^{n\lambda}\frac{dy}{t^{n}} w dx \frac{dt}{t} \\
&\lesssim \sum_{\substack{R \in \mathcal{D}_{good} \\ \ell(R) \leq 2^s}} \iint_{W_R}
\bigg[\sum_{\substack{\ell(R) \leq \ell(Q) \leq 2^r \ell(R) \\ d(Q,R) \leq \ell(R)^{\gamma} \ell(Q)^{1-\gamma}}}  \frac{\ell(R)^\alpha}{(\ell(R)+d(Q,R))^{n+\alpha}} \sigma(Q)^{1/2} \big\| \Delta_Q^{\sigma}f \big\|_{L^2(\sigma)}\bigg]^2
w dx \frac{dt}{t} \\
&\lesssim \mathscr{A}^2 \big\| f \big\|_{L^2(\sigma)}^2 .
\end{align*}

\subsection{The Case $\ell(Q) > 2^r \ell(R)$ and $d(Q,R) \leq \ell(R)^{\gamma} \ell(Q)^{1-\gamma}$.}
We call $R^{(k)}$ as the $k$ generations older dyadic ancestor of $R$. In this case, since $R$ is good, it must actually have
$R \subset Q$. That is, $Q$ is the ancestor of $R$. Then we can write
\begin{align*}
&\sum_{\substack{R \in \mathcal{D}_{good} \\ \ell(R) \leq 2^s}} \iint_{W_R} \int_{\Rn} \Big|
 \sum_{\substack{2^s \geq \ell(Q) > 2^r \ell(R) \\ d(Q,R) \leq \ell(R)^{\gamma} \ell(Q)^{1-\gamma}}}
\psi_t*(\Delta_Q^{\sigma}f \cdot \sigma)(x-y)\Big|^2 \Big(\frac{t}{t+|y|}\Big)^{n\lambda}\frac{dy}{t^{n}} w dx \frac{dt}{t} \\
&=\sum_{\substack{R \in \mathcal{D}_{good} \\ \ell(R) \leq 2^{s-r-1}}} \iint_{W_R} \int_{\Rn} \Big|\sum_{k=r+1}^{s-\log_2 \ell(R)} \psi_t*(\Delta_{R^{(k)}}^{\sigma}f \cdot \sigma)(x-y)\Big|^2 \Big(\frac{t}{t+|y|}\Big)^{n\lambda}\frac{dy}{t^{n}} w dx \frac{dt}{t} \\
&\leq \sum_{\substack{R \in \mathcal{D}_{good} \\ \ell(R) \leq 2^{s-r-1}}} \iint_{W_R} \int_{\Rn} \Big|\sum_{k=r+1}^{s-\log_2 \ell(R)} \psi_t*((\mathbf{1}_{R^{(k)}\setminus R^{(k-1)}}\Delta_{R^{(k)}}^{\sigma}f)\sigma)(x-y)\Big|^2 \Big(\frac{t}{t+|y|}\Big)^{n\lambda}\frac{dy}{t^{n}} w dx \frac{dt}{t} \\
&\quad + \sum_{\substack{R \in \mathcal{D}_{good} \\ \ell(R) \leq 2^{s-r-1}}} \iint_{W_R} \int_{\Rn} \Big|\sum_{k=r+1}^{s-\log_2 \ell(R)} \psi_t*((\mathbf{1}_{R^{(k-1)}}\Delta_{R^{(k)}}^{\sigma}f)\sigma)(x-y)\Big|^2 \Big(\frac{t}{t+|y|}\Big)^{n\lambda}\frac{dy}{t^{n}} w dx \frac{dt}{t} \\
&:= \mathcal{J} + \mathcal{K}.
\end{align*}
Fix the summing variable $k \geq r+1$. Then, the inequality $(\ref{R-K-S})$ implies that
\begin{align*}
&\sum_{\substack{R \in \mathcal{D}_{good} \\ \ell(R) \leq 2^{s-r-1}}} \iint_{W_R} \int_{\Rn} \Big|\psi_t*((\mathbf{1}_{R^{(k)}\setminus R^{(k-1)}}\Delta_{R^{(k)}}^{\sigma}f)\sigma)(x-y)\Big|^2 \Big(\frac{t}{t+|y|}\Big)^{n\lambda}\frac{dy}{t^{n}} w dx \frac{dt}{t} \\
&\lesssim 2^{-k \alpha} \sum_{\substack{R \in \mathcal{D}_{good} \\ \ell(R) \leq 2^{s-r-1}}} \mathcal{P}_\alpha(R^{(k)},|\Delta_{R^{(k)}}^{\sigma}f|\sigma)^2 w(R) \\
&\lesssim 2^{-k \alpha} \sum_{I} \big\| \Delta_{I}^{\sigma}f \big\|_{L^2(\sigma)}^2 \frac{\sigma(I)}{|I|} \frac{w(I)}{|I|}
\lesssim 2^{-k \alpha} \mathscr{A}^2 \big\| f \big\|_{L^2(\sigma)}^2,
\end{align*}
where we reindexed the sum over $R$ above. By the geometric decay in $k$, we deduce
$$ \mathcal{J} \lesssim \mathscr{A}^2 \big\| f \big\|_{L^2(\sigma)}^2. $$

It remains only to analyze the contribution made to $\mathcal{K}$ by the term$(\Delta_{R^{(k)}}^{\sigma}f) \mathbf{1}_{R^{(k-1)}}$. Our goal is to prove
\begin{equation}\label{estimate-K}
 \mathcal{K} \lesssim (\mathscr{A} + \mathscr{B})^2 \big\| f \big\|_{L^2(\sigma)}^2.
\end{equation}
To finish this, we here need an extra concept : $Stopping \ cubes$. For more applications and consequences associated with stopping cubes,  we  refer readers to the works \cite{Lacey1} ,\cite{Lacey2}, \cite{LSTS}.
The following argument is essentially taken from \cite{LL}.

\vspace{0.4cm}
\noindent\textbf{Stopping Cubes.}
We make the following construction of stopping cubes $\mathcal{S}$.
Let $\mathcal{D}_f$ be the dyadic children of good cubes $I \subset Q^0$ with $\log_2 \ell(I) = r' \text{ mod } r + 1$, where the integer $0 \leq r' < r+1$.
Set $\mathcal{S}_0$ to be all the maximal dyadic children of $Q_0$, which are in $\mathcal{D}_f$. Then set $\tau(S) = \mathbb{E}_S^{\sigma}f$, for $S \in \mathcal{S}_0$. In the recursive step, assuming that $\mathcal{S}_k$ is constructed, for $S \in \mathcal{S}_k$, set $\text{ch}_{\mathcal{S}}(S)$ to be the maximal subcubes $I \subset S$, $I \in \mathcal{D}_f$, such that either
\begin{enumerate}
\item [(a)] $\mathbb{E}_I^{\sigma}|f| > 2 \tau(S)$;
\item [(b)] The first condition fails, and
           $\sum_{K \in \mathcal{W}_I} \mathcal{P}_\alpha(K,\mathbf{1}_S \sigma)^2 w(K) \geq C_0 \mathscr{P}^2 \sigma(I)$.
\end{enumerate}
Then, define $\mathcal{S}_{k+1} := \bigcup_{S \in \mathcal{S}_k} \text{ch}_{\mathcal{S}}(S)$,
and for any $\dot{S} \in \text{ch}_{\mathcal{S}}(S)$
\begin{equation*}
\tau(\dot{S}) :=
\begin{cases}
\ \mathbb{E}_{\dot{S}}^{\sigma}|f| \ \ & \mathbb{E}_{\dot{S}}^{\sigma}|f| > 2 \tau(S), \\
\ \tau(S)   \ \ &\text{otherwise } .
\end{cases}
\end{equation*}
Finally, $\mathcal{S} := \bigcup_{k=0}^\infty \mathcal{S}_k$. Note that $\ell(\dot{S}) \leq 2^{-r-1}\ell(S)$ for all $\dot{S} \in \text{ch}_{\mathcal{S}}(S)$. In particular, it follows that
\begin{equation}\label{S-1-K}
\dot{S}^{(1)} \subset K, \ \text{for \ some} \ K \in \mathcal{W}_S.
\end{equation}
This holds since $\dot{S}^{(1)}$ is good, and strongly contained in $S$, so that Proposition $\ref{overlap}$ gives the implication above.

\vspace{0.4cm}
\noindent\textbf{Notation.}
For any dyadic cube $I$, $S(I)$ will denote its father in $\mathcal{S}$, the minimal cube in $\mathcal{S}$ that contains it. Note that there maybe the case $S(I) = I$. For any stopping cube $S$, $\mathscr{F}(S)$ will denote its father in the stopping tree, inductively, $\mathscr{F}^{k+1}S = \mathscr{F}(\mathscr{F}^k S)$.

The construction enjoys the following properties, which were proved in \cite{LSTS}.
\begin{lemma}\label{bound}
The following statements hold.
\begin{enumerate}
\item [(i)] For all cubes $I$, $|\mathbb{E}_{I}^{\sigma}f| \lesssim \tau(S(I))$.
\item [(ii)] The quasi-orthogonality bound holds :
\begin{equation}\label{quasi-orth}
\sum_{S \in \mathcal{S}} \tau(S)^2 \sigma(S) \lesssim \big\| f \big\|_{L^2(\sigma)}^2.
\end{equation}
\end{enumerate}
\end{lemma}

Applying the tool of stopping cubes, we can make the following decomposition.

\begin{equation}\aligned\label{stopping decomposition}
&\sum_{k=r+1}^{s-\log_2 \ell(R)}\mathbf{1}_{R^{(k-1)}} \Delta_{R^{(k)}}^{\sigma}f
=\sum_{k=r+1}^{s-\log_2 \ell(R)}(\mathbb{E}_{R^{(k-1)}}^{\sigma} \Delta_{R^{(k)}}^{\sigma}f) \mathbf{1}_{R^{(k-1)}} \\
&= \sum_{m=1}^{\infty} \sum_{k=r+1}^{s-\log_2 \ell(R)} \mathbf{1}_{\mathscr{F}^m S(R^{(r)}) \subset S(R^{(k-1)})} (\mathbb{E}_{R^{(k-1)}}^{\sigma} \Delta_{R^{(k)}}^{\sigma}f) \mathbf{1}_{\mathscr{F}^mS(R^{(r)}) \setminus \mathscr{F}^{m-1}S(R^{(r)})} \\
&\quad + \sum_{k=r+1}^{s-\log_2 \ell(R)}(\mathbb{E}_{R^{(k-1)}}^{\sigma} \Delta_{R^{(k)}}^{\sigma}f) \mathbf{1}_{S(R^{(r)})}
 - \sum_{k=r+1}^{s-\log_2 \ell(R)}(\mathbb{E}_{R^{(k-1)}}^{\sigma} \Delta_{R^{(k)}}^{\sigma}f) \mathbf{1}_{S(R^{(k-1)}) \setminus R^{(k-1)}}.
\endaligned
\end{equation}
Now, we are in the position to consider the contribution of $\mathcal{K}$, which is defined in the beginning of 5.4. Recall that
$$
\mathcal{K} = \sum_{\substack{R \in \mathcal{D}_{good} \\ \ell(R) \leq 2^{s-r-1}}} \iint_{W_R} \int_{\Rn} \Big|\sum_{k=r+1}^{s-\log_2 \ell(R)} \psi_t*((\mathbf{1}_{R^{(k-1)}}\Delta_{R^{(k)}}^{\sigma}f)\sigma)(x-y)\Big|^2 \Big(\frac{t}{t+|y|}\Big)^{n\lambda}\frac{dy}{t^{n}} w dx \frac{dt}{t} .
$$
Thus, $\mathcal{K}$ is bounded by corresponding three parts, which are written as $\mathcal{K}_{\text{Glo}}$, $\mathcal{K}_{\text{Par}}$ and $\mathcal{K}_{\text{Loc}}$ respectively. We next shall estimate each one successively.

\vspace{0.5cm}
\noindent\textbf{$\bullet$ The Global Part.}
First, we analyze the first term on the right of $(\ref{stopping decomposition})$. It is worth noting that reindexing the corresponding sum is crucial to our estimates. To do this, we here borrow an idea from \cite{LL}.

Fix a stopping cube $S$ and integer $m$. Consider $\ddot{S} \in \mathcal{S}$, and split integer $m = p + q$, where $p = \lceil m/2 \rceil$. Consider the sub-partition of $\ddot{S}$ given by
$\mathcal{P}(m,\ddot{S}) = \{\dot{S} \in \mathcal{S }: \mathscr{F}^p \dot{S} = \ddot{S}\}$.
Now, for stopping cube $S$ with $\mathscr{F}^q S = \dot{S}$ , and good $R \Subset \dot{S}$, we have $R \subset \dot{K}$ for some
$\dot{K} \in \mathcal{W}_{\dot{S}}$ , where $\dot{S} \in \mathcal{P}(m,\ddot{S})$. Note that we have
$R \subset \dot{K} \subset \dot{S}$. It follows from the goodness of $R$ that he assumption of of Lemma $\ref{lemma-R-K-S}$ holds for these three intervals.
The above argument is saying that
$$
\bigcup_{\dot{S} \in \mathcal{P}(m,\ddot{S})} \bigcup_{\substack{R:good, R \Subset \dot{S} \\ \mathscr{F}^q S = \dot{S}}}R
\subset \bigcup_{\dot{S} \in \mathcal{P}(m,\ddot{S})} \bigcup_{\dot{K} \in \mathcal{W}_{\dot{S}}} \bigcup_{R:R \subset \dot{K}}R .
$$
In addition, one can find a constant $|c| \lesssim 1$ such that
\begin{equation*}
\sum_{k=r+1}^{s-\log_2 \ell(R)} \mathbf{1}_{\mathscr{F}^m S \subset S(R^{(k-1)})} \mathbb{E}_{R^{(k-1)}}^{\sigma} \Delta_{R^{(k)}}^{\sigma}f
=c \cdot \tau(\mathscr{F}^m S) .
\end{equation*}

Thereby, for each $\ddot{S} \in \mathcal{S}$, using the above facts, we obtain
\begin{align*}
\Lambda(\ddot{S})&:=\sum_{\substack{R:\mathscr{F}^m S=\ddot{S} \\ S=S(R^{(r)})}}  \iint_{W_R} \int_{\Rn} \Big|\sum_{\substack{k=r+1 \\ \ddot{S} \subset S(R^{(k-1)})}}^{s-\log_2 \ell(R)} \mathbb{E}_{R^{(k-1)}}^{\sigma}\Delta_{R^{(k)}}^{\sigma}f \cdot
\psi_t*( \mathbf{1}_{\ddot{S} \setminus \mathscr{F}^{m-1}S} \sigma)(x-y)\Big|^2 \\
&\quad\quad\quad \times \Big(\frac{t}{t+|y|}\Big)^{n\lambda}dy \frac{w dx dt}{t^{n+1}} \\
&\lesssim \tau(\ddot{S})^2 \sum_{\dot{S} \in \mathcal{P}(m,\ddot{S})} \sum_{\dot{K} \in \mathcal{W}_{\dot{S}}}
\sum_{R: R \subset \dot{K}} \int_{\Rn} \iint_{W_R}|\psi_t*(\mathbf{1}_{\ddot{S} \setminus \mathscr{F}^{p-1}\dot{S}} \sigma)(x-y)|^2 \Big(\frac{t}{t+|y|}\Big)^{n\lambda} w dx \frac{dt dy}{t^{n+1}}.
\end{align*}
Furthermore, from $(\ref{R-K-S})$ and Proposition $\ref{best constant}$, it follows that
\begin{equation}\aligned\label{Lambda-S}
\Lambda(\ddot{S})&\lesssim \tau(\ddot{S})^2 \sum_{\dot{S} \in \mathcal{P}(m,\ddot{S})} \sum_{\dot{K} \in \mathcal{W}_{\dot{S}}}\mathcal{P}_\alpha(\dot{K},\mathbf{1}_{\ddot{S}} \sigma)^2 \sum_{R: R \subset \dot{K}} \bigg(\frac{\ell(R)}{\ell(\dot{K})}\bigg)^\alpha w(R) \\
&\lesssim 2^{-m \alpha /2} \tau(\ddot{S})^2 \sum_{\dot{S} \in \mathcal{P}(m,\ddot{S})} \sum_{\dot{K} \in \mathcal{W}_{\dot{S}}}\mathcal{P}_\alpha(\dot{K},\mathbf{1}_{\ddot{S}} \sigma)^2 w(\dot{K}) \\
&\lesssim 2^{-m \alpha /2} (\mathscr{A} + \mathscr{B})^2 \tau(\ddot{S})^2 \sigma(\ddot{S}).
\endaligned
\end{equation}

Now we turn our attention to bound $\mathcal{K}_{\text{Glo}}$. Making use of Cauchy-Schwartz inequality and $(\ref{Lambda-S})$, we deduce that
\begin{align*}
\mathcal{K}_{\text{Glo}}&=\sum_{\substack{R \in \mathcal{D}_{good} \\ \ell(R) \leq 2^{s-r-1}}} \iint_{W_R} \int_{\Rn} \Big|\sum_{m=1}^\infty 2^{-m \alpha/8} 2^{m \alpha/8} \sum_{\substack{k=r+1 \\ \mathscr{F}^m S(R^{(r)}) \subset S(R^{(k-1)})}}^{s-\log_2 \ell(R)} \mathbb{E}_{R^{(k-1)}}^{\sigma}\Delta_{R^{(k)}}^{\sigma}f \\
&\quad\quad \times \psi_t*(\mathbf{1}_{\mathscr{F}^mS(R^{(r)}) \setminus \mathscr{F}^{m-1}S(R^{(r)})} \sigma)(x-y)\Big|^2 \Big(\frac{t}{t+|y|}\Big)^{n\lambda} w dx \frac{dt dy}{t^{n+1}} \\
&\lesssim \sum_{m=1}^\infty 2^{m \alpha/4} \sum_{\substack{R \in \mathcal{D}_{good} \\ \ell(R) \leq 2^{s-r-1}}} \iint_{W_R} \int_{\Rn} \Big| \sum_{k=r+1}^{s-\log_2 \ell(R)} \mathbb{E}_{R^{(k-1)}}^{\sigma}\Delta_{R^{(k)}}^{\sigma}f \\
&\quad\quad \times \psi_t*(\mathbf{1}_{\mathscr{F}^mS(R^{(r)}) \setminus \mathscr{F}^{m-1}S(R^{(r)})} \sigma)(x-y)\Big|^2 \Big(\frac{t}{t+|y|}\Big)^{n\lambda} w dx \frac{dt dy}{t^{n+1}} \\
&\leq \sum_{m=1}^\infty 2^{m \alpha/4} \sum_{\ddot{S} \in \mathcal{S}} \Lambda(\ddot{S})
\lesssim (\mathscr{A} + \mathscr{B})^2 \sum_{m=1}^\infty 2^{-m \alpha /4} \sum_{\ddot{S} \in \mathcal{S}} \tau(\ddot{S})^2 \sigma(\ddot{S})
\lesssim (\mathscr{A} + \mathscr{B})^2 \big\| f \big\|_{L^2(\sigma)}^2,
\end{align*}
where in the last step, we used the quasi-orthogonality bound $(\ref{quasi-orth})$.

Let us next explain how to obtain the geometric factor in $(\ref{Lambda-S})$. We can assume that $q > 2$. Now, $S(R) \subset S$ and
$\mathscr{F}^q S = \dot{S}$. Write the stopping cubes between $S$ and $\dot{S}$ as
$$R \subset S = S_1 \subsetneq S_2 \subsetneq \cdots \subsetneq S_q := \dot{S},\ \ S_t \in \mathcal{S},\ 1 \leq t \leq q.$$
Observing $(\ref{S-1-K})$, we have $S_{q-1} \subset \dot{K}$, for $\dot{K} \in \mathcal{W}_{\dot{S}}$ as above. Then, we have
$\ell(R) \leq 2^{-q+1}\ell(\dot{K})$. Since $q \simeq m/2$, we obtain the geometric decay in $m$ above.

\vspace{0.5cm}
\noindent\textbf{$\bullet$ The Paraproduct Estimate.}
Next, we bound the second term on the right of $(\ref{stopping decomposition})$. It is worth noting that the sum over the martingale differences is controlled by the stopping value $\tau(S)$. That is,
\begin{align*}
\Big|\sum_{k=r+1}^{s-\log_2 \ell(R)} \mathbb{E}_{R^{(k-1)}}^{\sigma}\Delta_{R^{(k)}}^{\sigma}f \Big|
=\big| \mathbb{E}_{R^{(r)}}^{\sigma} f \big|
\lesssim \tau(R^{(r)}).
\end{align*}

Therefore, for fixed $S \in \mathcal{S}$, the testing condition $(\ref{testing condition-g^star})$ implies that
\begin{align*}
\Gamma(S)&:=
\sum_{\substack{R:\ell(R) \leq 2^{s-r-1} \\ S(R^{(r)})=S}}  \iint_{W_R} \int_{\Rn} \Big|\sum_{k=r+1}^{s-\log_2 \ell(R)} \psi_t*(\mathbb{E}_{R^{(k-1)}}^{\sigma}\Delta_{R^{(k)}}^{\sigma}f \cdot \mathbf{1}_S \sigma)(x-y)\Big|^2 \Big(\frac{t}{t+|y|}\Big)^{n\lambda}\frac{dy}{t^{n}} w dx \frac{dt}{t} \\
&\lesssim \tau(S)^2 \int_{\Rn} \iint_{\widehat{S}}|\psi_t*(\mathbf{1}_S \sigma)(x-y)|^2 \Big(\frac{t}{t+|y|}\Big)^{n\lambda} w dx \frac{dt dy}{t^{n+1}}
\lesssim \mathscr{B}^2 \tau(S)^2 \sigma(S).
\end{align*}

Accordingly, by the quasi-orthogonality bound $(\ref{quasi-orth})$, it yields that
\begin{align*}
\mathcal{K}_{\text{Par}}
\lesssim \sum_{S \in \mathcal{S}} \Gamma(S)
\lesssim \mathscr{B}^2 \sum_{S \in \mathcal{S}} \tau(S)^2 \sigma(S)
\lesssim \mathscr{B}^2 \big\|f\big\|_{L^2(\sigma)}^2.
\end{align*}

\vspace{0.5cm}
\noindent\textbf{$\bullet$ The Local Bound.}
Finally, let us estimate the third term on the right of $(\ref{stopping decomposition})$. Fix an integer $k \geq r+1$ and fix a (good) cube $\dot{R}$, and child $\ddot{R}$ of $\dot{R}$. Observe that if $\ddot{R}=R^{(k-1)}$, then $R \Subset \ddot{R}$. Then by Proposition $\ref{overlap} ~ (1)$, there exists a cube $K \in \mathcal{W}_{\ddot{R}}$ such that $R \subset K$.

For such $K$, from $(\ref{R-K-S})$, it follows that
\begin{align*}
\Theta(\ddot{R},K)&:=
\sum_{\substack{R: R \subset K \\ R^{(k-1)}=\ddot{R}}} \iint_{W_R} \int_{\Rn} |\psi_t*(\mathbf{1}_{S(\ddot{R}) \setminus \ddot{R}} \sigma)(x-y)|^2 \Big(\frac{t}{t+|y|}\Big)^{n\lambda}\frac{dy}{t^{n}} w dx \frac{dt}{t} \\
&\lesssim 2^{-k \alpha /2} \mathcal{P}_\alpha(K,\mathbf{1}_{S(\ddot{R})} \sigma)^2 w(K).
\end{align*}

We will see that the stopping rule on the pivotal condition plays an important role. Indeed, if $\ddot{R}$ is a stopping cube, then
$S(\ddot{R})=\ddot{R}$. Hence, stopping cube $\ddot{R}$ does not contribute at all to our summation. This leads us to only consider
non-stopping cube $\ddot{R}$ below.

Because $\ddot{R}$ is not a stopping cube, it must fail the pivotal stopping condition. Then by Proposition $\ref{best constant}$,
we have that
\begin{align*}
\sum_{K \in \mathcal{W}_{\ddot{R}}} \Theta(\ddot{R},K)
\lesssim 2^{-k \alpha /2} \sum_{K \in \mathcal{W}_{\ddot{R}}} \mathcal{P}_\alpha(K,\mathbf{1}_{S(\ddot{R})} \sigma)^2 w(K)
\lesssim 2^{-k \alpha /2} (\mathscr{A} + \mathscr{B})^2  \sigma(\ddot{R}).
\end{align*}
Therefore, together with the above estimates, it yields that
\begin{align*}
\mathcal{K}_{\text{Loc}}
&=\sum_{\substack{R \in \mathcal{D}_{good} \\ \ell(R) \leq 2^{s-r-1}}} \iint_{W_R} \int_{\Rn} \Big| \sum_{k=r+1}^{s-\log_2 \ell(R)} 2^{-k \alpha/6} 2^{k \alpha/6}\mathbb{E}_{R^{(k-1)}}^{\sigma}\Delta_{R^{(k)}}^{\sigma}f \psi_t*(\mathbf{1}_{S(R^{(k-1)}) \setminus R^{(k-1)}} \sigma)(x-y)\Big|^2 ... \\
&\lesssim \sum_{k=r+1}^{s-\log_2 \ell(R)} 2^{k \alpha/3} \sum_{\substack{R \in \mathcal{D}_{good} \\ \ell(R) \leq 2^{s-r-1}}}
\big| \mathbb{E}_{R^{(k-1)}}^{\sigma}\Delta_{R^{(k)}}^{\sigma}f \big|^2 \iint_{W_R} \int_{\Rn} |\psi_t*(\mathbf{1}_{S(R^{(k-1)}) \setminus R^{(k-1)}} \sigma)(x-y)|^2 ... \\
&\leq \sum_{k=r+1}^{s-\log_2 \ell(R)} 2^{k \alpha/3} \sum_{\ddot{R} \notin \mathcal{S}}
\big| \mathbb{E}_{\ddot{R}}^{\sigma}\Delta_{\dot{R}}^{\sigma}f \big|^2 \sum_{K \in \mathcal{W}_{\ddot{R}}} \Theta(\ddot{R},K) \\
&\lesssim (\mathscr{A} + \mathscr{B})^2 \sum_{k=r+1}^{s-\log_2 \ell(R)} 2^{-k \alpha/6} \sum_{\ddot{R} \notin \mathcal{S}} \big| \mathbb{E}_{\ddot{R}}^{\sigma}\Delta_{\dot{R}}^{\sigma}f \big|^2 \sigma(\ddot{R})\\
&\lesssim  (\mathscr{A} + \mathscr{B})^2 \big\| f \big\|_{L^2(\sigma)}^2.
\end{align*}

So far, we have proved $(\ref{estimate-K})$. Consequently, we complete the proof of sufficiency in Theorem $\ref{Theorem g^star}$.

\qed
\subsection*{Acknowledgements}
The authors want to express their sincerely thanks to the referee
for his or her valuable remarks and suggestions which made this
paper more readable.

\end{document}